\documentclass[final,11pt]{elsarticle}

\usepackage{verbatim,a4wide}

\usepackage{amsmath}
\usepackage{amssymb}
\usepackage{amsthm}

\usepackage{xcolor}

\usepackage{hyperref}

\usepackage[cp1250]{inputenc}
\usepackage[OT4]{fontenc}
\usepackage[english]{babel}

\usepackage{datetime}

\numberwithin{equation}{section}
\numberwithin{figure}{section}
\numberwithin{table}{section}

\newtheorem{theorem}{Theorem}[section]
\newtheorem{lemma}[theorem]{Lemma}
\newtheorem{remark}[theorem]{Remark}
\newtheorem{corollary}[theorem]{Corollary}

\newtheorem{assumption}[theorem]{Assumption}

\begin{document}

\begin{frontmatter}

\title{Dual Gauss--Legendre polynomials}

\author[A1]{Pawe{\l} Wo\'{z}ny\corref{cor}}
\ead{Pawel.Wozny@cs.uni.wroc.pl}

\cortext[cor]{Corresponding author.
     \hfill{\color{red}\fbox{{\color{black} Version:} \today, \currenttime}}}

\address[A1]{Institute of Computer Science, University of Wroc{\l}aw,
             ul.~Joliot-Curie 15, 50-383 Wroc{\l}aw, Poland}

\begin{abstract}
We define and investigate two families of dual polynomials associated with the 
Gauss--Legendre polynomials, which have recently found interesting applications 
in computer graphics. Using the presented results, one can derive
representations of the Gauss--Legendre polynomials, construct the dual bases 
for Lagrange bases and solve certain approximation problems arising, for 
example, in CAGD.
\end{abstract}

%
%
%
%


\end{frontmatter}

\section{Introduction}                                  \label{S:Introduction}

The \textit{Gauss--Legendre polynomials} (\textit{GL polynomials} for short) 
and the \textit{Gauss--Legendre curves} (\textit{GL curves} for short) 
associated with them were recently introduced by Moon et al.~in~\cite{Moon2023} 
(see also~\cite{Kim2026,Moon2026,RH2026}). Both GL polynomials 
and GL curves possess several interesting properties and therefore constitute 
attractive tools with numerous potential applications in numerical analysis and 
computer graphics. For example, GL curves have very good shape control 
properties. However, they do not lie within the convex hull of their control 
points, but they closely follow the control polygon, even at high degrees.  

In the sequel, $\Pi_k$ denotes the space of univariate polynomials of degree 
at most $k$ ($\Pi_{-1}:=\emptyset$), while $\delta_{ij}$ is the 
\textit{Kronecker delta} ($\delta_{ii}:=1$, $\delta_{ij}:=0$ for $i\neq j$).

The GL polynomials are related to the classical \textit{Legendre polynomials} 
$P_k\in\Pi_{k}\setminus\Pi_{k-1}$ $(k\in\mathbb N)$ which are orthogonal on 
$[-1,1]$ with respect to the $L^2$ inner product, i.e.,
$$
\int_{-1}^{1}P_i(x)P_j(x)\,\mathrm{d}x=\delta_{ij}\frac{2}{2i+1}\qquad
                                                             (i,j=0,1,\ldots),
$$
and satisfy the normalization condition $P_k(1)=1$ $(k=0,1,\ldots)$. See, e.g., 
\cite{KLS2010,STW2011}.

\begin{assumption}\label{A:Assumption-1}
From now on, let $n\in\mathbb N$ be fixed. Let $\tau_i\equiv\tau^{(n)}_i$ 
$(i=1,2,\ldots,n;\ n>0)$,
\begin{equation*}
-1<\tau_1<\tau_2<\cdots<\tau_n<1,
\end{equation*}
be the \textit{zeros} of the $n$th Legendre polynomial, i.e., 
$P_n(\tau_i)=0$ $(1\leq i\leq n)$.

In this paper, we assume that the zeros $\tau^{(n)}_i$ $(i=1,2,\ldots,n)$, as
well as the {\rm Gauss--Legendre weights}, 
\begin{equation}\label{E:Def-omega_i}
\omega^n_i\equiv\omega_i:=\frac{2}{nP_{n-1}(\tau_i)P_n'(\tau_i)}
                                                       \qquad(i=1,2,\ldots,n),
\end{equation}
are known with high numerical accuracy, since they can be precomputed using, 
for example, the results presented in~\cite{HT2013,JM2018}. 
\end{assumption}

\begin{remark}\label{R:GL-quadrature}
Recall that 
\begin{equation}\label{E:GL-quadrature}
Q^{GL}_n(f):=\sum_{i=1}^{n}\omega_if(\tau_i)
             \approx\int_{-1}^{1}f(x)\,\mathrm{d}x
\end{equation}
is the {\rm Gauss--Legendre quadrature}, which is exact for all polynomials
of degree at most $2n-1$, i.e.,
$$
\int_{-1}^{1}p(x)\,\mathrm{d}x=Q^{GL}_n(p)
$$
for every $p\in\Pi_{2n-1}$ (see, e.g., \cite{WG}).
\end{remark}

The GL polynomials $F^n_i\in\Pi_{n}\setminus\Pi_{n-1}$ $(i=0,1,\ldots,n)$ are 
given by
\begin{equation}\label{E:Def-F^n_i}
F^n_i(t):=G^n_i(t)-G^n_{i+1}(t),
\end{equation}
where $G^n_0(t)=-G^n_{n+1}(t)\equiv\frac{1}{2}$, and
\begin{equation}\label{E:Def-G^n_i}
G^n_i(t):=\frac{nP_{n-1}(\tau_i)}{2}
               \int_{-1}^{t}\frac{P_n(x)}{x-\tau_i}\,\mathrm{d}x-\frac{1}{2}
\end{equation}
for $i=1,2,\ldots,n$.

For efficient methods for the evaluation of GL polynomials and their linear 
combinations (in particular, $d$-dimensional GL curves), we refer the reader 
to~\cite{ChW2026} (see also \cite{RH2026}).

In this paper, we define and study the \textit{dual polynomials} for the 
polynomials $F^n_i$ and $G^n_i$ with respect to the Jacobi 
inner product. The resulting dual systems are expected to be useful in 
computer-aided geometric design (CAGD) and numerical analysis, particularly in 
some approximation problems.

This work is motivated by the context outlined in~\cite{Moon2023}, where GL 
polynomials were shown to inherit several structural properties of Bernstein 
polynomials, which play a fundamental role in approximation theory and computer 
graphics. In a related direction, \textit{dual Bernstein polynomials} 
(see~\cite{LW2006})---defined via duality with respect to the (shifted) Jacobi 
inner product---have proved to be a powerful tool, for example in addressing 
the \textit{degree reduction problem} for parametric \textit{B\'{e}zier curves} 
(see~\cite{WL2009}), which are central objects in CAD and CAGD (see, e.g., 
\cite{Agoston2005,Farin2002}).

The paper is organized as follows. Section~\ref{S:DualBases} contains a brief 
overview of the theory and applications of dual bases. In 
Section~\ref{S:New-dual}, which constitutes the main part of the paper, we 
derive explicit representations for the dual $G^n_i$ polynomials as well as for 
the dual GL polynomials. Finally, Section~\ref{S:Applications} is devoted to 
selected applications of the results presented in this article.

\section{Dual bases}                                       \label{S:DualBases}

Let $B_m:=\{b^m_0, b^m_1,\ldots,b^m_m\}$ $(m\in\mathbb N)$ be a set of linearly 
independent functions. Consider the space ${\cal B}_m:=\operatorname{span}B_m$ 
with an inner product 
$\left<\cdot,\cdot\right>:{\cal B}_m\times{\cal B}_m\to\mathbb R$. Clearly,
$m+1$ is the dimension of ${\cal B}_m$. 

The set $D_m:=\{d^m_0,d^m_1,\ldots,d^m_m\}\subset{\cal B}_m$ is called 
a \textit{dual basis} for the basis $B_m$ of the space ${\cal B}_m$ with 
respect to the inner product $\left<\cdot,\cdot\right>$ if and only if
$$
{\cal D}_m:=\mbox{span}\;D_m={\cal B}_m\qquad\mbox{and}\qquad 
                               \left<b^m_i,d^m_j\right>=\delta_{ij}
                                                      \quad(0\leq i,j\leq m).
$$

The dual bases are commonly used, for example, in approximation theory, 
numerical analysis, and CAGD. We now outline their two main properties.

\begin{enumerate}
\itemsep 1ex

\item \textit{Representation.} Any function $f\in{\cal B}_m$ can be represented 
in the basis $B_m$ as follows:
$$
f=\sum_{k=0}^{m}\left<f,d^m_k\right>b^m_k.
$$

\item \textit{Approximation.} For a given function $g$, an 
\textit{optimal element} $p^\ast\in{\cal B}_m$ in the sense of the 
\textit{least-squares} approximation has the form
$$
p^\ast=\sum_{k=0}^{m}\left<g,d^m_k\right>b^m_k,
$$
which means
$$
||g-p^\ast||=\min_{p\in{\cal B}_m}||g-p||,
$$
where $||\cdot||:=\sqrt{\left<\cdot,\cdot\right>}$. 

\end{enumerate}

Thus, the dual basis allows us to find the optimal element without using an 
orthogonal basis. This is very attractive when we want to solve approximation 
problems in a specific, non-orthogonal basis, which is often necessary. Note 
that the incorporation of dual bases can very often result in algorithms with 
lower computational complexity. For these reasons, methods for constructing 
dual bases have been intensively studied in recent years. It should also be 
mentioned that, for many important bases, the corresponding dual bases are 
known explicitly, particularly in the polynomial case, which will be considered 
in the following paragraph.

For efficient methods for construction of the dual bases and their further 
properties and applications, see, e.g., \cite{PW2013,PW2014}, as 
well as \cite{BJ1998,BJ2007,ZC1987,LW2006b,LW2006,LW2011,LW2011b,LWK2011,
LZLW2009,RN2007,RN2008,W2012,WL2009,WL2010} (dual Bernstein polynomials), 
\cite{RNG} (dual polynomial bases), \cite{GXZ06} (dual B-spline functional), 
\cite{Zhang2010,Zhang2009a,Zhang2009b} (dual Wang-B\'{e}zier and 
B\'{e}zier-Said-Wang type generalized Ball polynomials), \cite{Zhang2009c} 
(dual NS power basis) and the references cited therein.

\subsection{Dual polynomial bases}                    \label{SS:DualPolyBases}

Let us consider the polynomial case, i.e., ${\cal B}_m=\Pi_m$. Let 
$Q_m:=\{q_0,q_1,\ldots,q_m\}$ be the set of orthogonal polynomials with respect 
to the inner product $\left<\cdot,\cdot\right>$. This means that
\begin{equation}\label{E:Orthogonality}
\left<q_i,q_j\right>=\delta_{ij}h_i\qquad (i,j=0,1,\ldots,m;\;h_i>0),
\end{equation}
and $q_i\in\Pi_{i}\setminus\Pi_{i-1}$ $(0\leq i\leq m)$.

Using the following theorem, one can construct the dual basis of $B_m$ in terms 
of an orthogonal basis of $\Pi_m$.

\begin{theorem}[see Lemma 2.1 in \cite{LW2006}]\label{T:Thm-LW2006-I}
Let $c_{ij}$ be the coefficients in
\begin{equation*}
q_i=\sum_{j=0}^{m}c_{ij}b^m_j\qquad (i=0,1,\ldots,m).
\end{equation*}
Then the elements $d^{m}_i$ $(0\leq i\leq m)$ of the dual basis $D_m$ are 
given by
$$
d^{m}_i=\sum_{j=0}^{m}h_j^{-1}c_{ji}q_j\qquad (i=0,1,\ldots,m).
$$
\end{theorem}

We also have an identity connecting the three families of polynomials, namely 
the elements of the basis $B_m$, the dual basis $D_m$, and the orthogonal basis 
$Q_{m+1}$.

\begin{theorem}[see Lemma 2.2 in \cite{LW2006}]\label{T:Thm-LW2006-II}
The following identity holds:
\begin{equation*}
k_m(x,y)=\sum_{i=0}^{m}b^m_i(x)d^m_i(y),
\end{equation*}
where $k_m(x,y)$ is the {\rm Christoffel--Darboux kernel},
\begin{equation}\label{E:Ch-D-kernel}
k_m(x,y):=\frac{\ell_m}{\ell_{m+1}h_m}
                           \frac{q_{m+1}(x)q_m(y)-q_m(x)q_{m+1}(y)}{x-y}.
\end{equation}
Here $\ell_m$ and $\ell_{m+1}$ denote the leading coefficients of the 
polynomials $q_m$ and $q_{m+1}$, respectively (see 
	also~\eqref{E:Orthogonality}).
\end{theorem}

\section{New dual bases}                                    \label{S:New-dual}

We will use Theorem~\ref{T:Thm-LW2006-I} to derive the dual polynomials for 
the polynomials $G^n_i$ and $F^n_i$ with respect to the \textit{Jacobi inner
product},
\begin{equation}\label{E:JacobiInnerP}
\left<f,g\right>_{\alpha,\beta}:=\int_{-1}^{1}(1-x)^\alpha(1+x)^\beta
f(x)g(x)\,\mathrm{d}x\qquad(\alpha,\beta>-1).
\end{equation}
Before doing so, however, a lemma and a corollary are also needed.

\begin{lemma}\label{L:Base-G}
The polynomials $G^n_0, G^n_1,\ldots, G^n_n$ (see~\eqref{E:Def-G^n_i}) form 
a basis of $\Pi_n$. Moreover, any polynomial $p\in\Pi_n$ can be represented in 
this basis as follows:
$$
p(t)=(p(-1)+p(1))G^n_0(t)+\sum_{i=1}^np'(\tau_i)\omega_iG^n_i(t)
$$
(see~Assumption~\ref{A:Assumption-1}).
\end{lemma}
\begin{proof}
Let us define 
\begin{equation}\label{E:Def-L^n_i} 
L^n_i(x):=\prod_{j=1 \atop j\neq i}^{n}\frac{x-\tau_j}{\tau_i-\tau_j}
\end{equation}
for $i=1,2,\ldots,n$. Then 
\begin{equation}\label{E:Omega_i-L^n_i}
\omega_i=\int_{-1}^{1}L^n_i(x)\,\mathrm{d}x\qquad (1\leq i\leq n)
\end{equation}
(cf.~\eqref{E:Def-omega_i} and Remark~\ref{R:GL-quadrature}). Obviously, 
for every polynomial $p\in\Pi_n$, we also have
$$
\Pi_{n-1}\ni p'(x)=\sum_{i=1}^{n}p'(\tau_i)L^n_i(x).
$$

Observe that
\begin{equation}\label{E:L^n_i-G^n_i}
L^n_i(x)=\omega_i\frac{nP_{n-1}(\tau_i)}{2}\frac{P_n(x)}{x-\tau_i}
\end{equation}
(cf.~\eqref{E:Def-G^n_i} and the proof of \cite[Lemma 5]{Moon2023}) and thus
$$
p'(x)=\sum_{i=1}^{n}p'(\tau_i)\omega_i\frac{nP_{n-1}(\tau_i)}{2}
                                                       \frac{P_n(x)}{x-\tau_i}.
$$

Integrating the above identity over $x\in[-1,t]$, where $-1\leq t\leq 1$, 
yields
$$
p(t)-p(-1)=\sum_{i=1}^{n}p'(\tau_i)\omega_i\left(G^n_i(t)+\frac{1}{2}\right),
$$
but
\begin{eqnarray*}
\frac{1}{2}\sum_{i=1}^{n}p'(\tau_i)\omega_i&=&
        \frac{1}{2}\int_{-1}^{1}
                  \left(\sum_{i=1}^{n}p'(\tau_i)L^n_i(x)\right)\mathrm{d}x\\
        &=&\frac{1}{2}\int_{-1}^{1}p'(x)\,\mathrm{d}x=
                                        \frac{1}{2}\left(p(1)-p(-1)\right)
\end{eqnarray*}
(see~\eqref{E:Omega_i-L^n_i}) and hence the result.
\end{proof}

Taking into account that
$$
G^n_i(t)+\frac{1}{2}=\sum_{j=i}^{n}F^n_j(t)\qquad (1\leq i\leq n)
$$
(cf.~\eqref{E:Def-F^n_i}), as well as using the \textit{partition of unity 
property} for GL polynomials (\cite[\S5.1]{Moon2023}),
$$
\sum_{i=0}^{n}F^n_i(t)\equiv 1\qquad (t\in\mathbb R),
$$
we obtain the following corollary, which is in agreement 
with~\cite[Theorem 7]{Moon2023}.

\begin{corollary}\label{C:Corollary-1}
If $p\in\Pi_n$ then
$$
p(t)=\sum_{i=0}^{n}\left(p(-1)+\sum_{j=1}^{i}p'(\tau_j)\omega_j\right)F^n_i(t).
$$
\end{corollary}

\subsection{Jacobi polynomials}                              \label{SS:Jacobi}

Recall that \textit{Jacobi polynomials} 
$P^{(\alpha,\beta)}_k\in\Pi_{k}\setminus\Pi_{k-1}$ $(k\in\mathbb N)$ are 
orthogonal with respect to the inner product~\eqref{E:JacobiInnerP}, i.e.,
\begin{equation}\label{E:JacobiOrtho}
\left<P^{(\alpha,\beta)}_k,P^{(\alpha,\beta)}_\ell\right>_{\alpha,\beta}=
\delta_{k\ell}h^{(\alpha,\beta)}_{k}\qquad(k,\ell=0,1,\ldots),
\end{equation}
where
\begin{equation}\label{E:Def-h_k}
h^{(\alpha,\beta)}_{k}:=2^\sigma\frac{\Gamma(k+\alpha+1)\Gamma(k+\beta+1)}
                                    {k!(2k+\sigma)\Gamma(k+\sigma)}
                                                         \qquad(k\in\mathbb N)
\end{equation}
with $\sigma:=\alpha+\beta+1$. 

The \textit{leading coefficient} of $P^{(\alpha,\beta)}_{k}$ is
\begin{equation}\label{E:Jacobi-lc}
\ell^{(\alpha,\beta)}_k=\frac{(k+\sigma)_k}{k!2^k}\qquad (k=0,1,\ldots),
\end{equation}
where $(z)_i$ $(z\in\mathbb R;\ i\in\mathbb N)$ denotes the \textit{Pochhammer
symbol},
$$
(z)_0:=1,\qquad (z)_i:=z(z+1)\cdots(z+i-1)\qquad (i=1,2,\ldots).
$$
We also have 
\begin{equation}\label{E:Jacobi+1-1}
P^{(\alpha,\beta)}_k(1)=\frac{(\alpha+1)_k}{k!},\qquad
P^{(\alpha,\beta)}_k(-1)=(-1)^k\frac{(\beta+1)_k}{k!},
\end{equation}
and
\begin{equation}\label{E:DiffJacobi}
\frac{\mathrm{d}}{\mathrm{d}x}P^{(\alpha,\beta)}_k(x)=
                      \frac{k+\sigma}{2}P^{(\alpha+1,\beta+1)}_{k-1}(x),                      
\end{equation}
where $k=0,1,\ldots$ Here, and in the sequel, we adopt the convention that 
\begin{equation}\label{E:Convention}
P^{(\alpha,\beta)}_{-1}(x)\equiv0.
\end{equation} 

Clearly, the Legendre polynomials belong to the Jacobi polynomials family, 
since $P_k(x)\equiv P^{(0,0)}_k(x)$ for $k\in\mathbb N$.

For further properties and applications of Jacobi polynomials, see, e.g., 
\cite{KLS2010,STW2011}. 

\subsection{Dual \texorpdfstring{$G^n_j$}{Gnj} polynomials}             
                                                    \label{SS:G^n_j-dual-poly}

The \textit{dual $G^n_j$ polynomials} $dG^n_j(t;\alpha,\beta)\in\Pi_n$ 
$(0\leq j\leq n;\,\alpha,\beta>-1)$ are defined to satisfy the following
conditions:
$$
\left<G^n_i,dG^n_j(\cdot;\alpha,\beta)\right>_{\alpha,\beta}=
                                          \delta_{ij}\qquad (0\leq i,j\leq n).
$$

To derive an explicit representation of the dual polynomials 
$dG^n_j(t;\alpha,\beta)$, we apply Theorem~\ref{T:Thm-LW2006-I}, where the 
Jacobi polynomials---as satisfying the orthogonality 
relation~\eqref{E:JacobiOrtho}---are expressed in terms of the 
polynomials~\eqref{E:Def-G^n_i},
\begin{eqnarray*}
P^{(\alpha,\beta)}_k(t)&=&
    \left(P^{(\alpha,\beta)}_{k}(-1)+P^{(\alpha,\beta)}_{k}(1)\right)G^n_0(t)+
        \sum_{i=1}^{n}\omega_i\frac{\mathrm{d}}{\mathrm{d}x}
                        P^{(\alpha,\beta)}_{k}(x)\Big|_{x:=\tau_i}G^n_i(t)\\
&=&\frac{1}{k!}((\alpha+1)_k+(-1)^k(\beta+1)_k)G^n_0(t)+
                \frac{k+\sigma}{2}\sum_{i=1}^{n}
                        \omega_iP^{(\alpha+1,\beta+1)}_{k-1}(\tau_i)G^n_i(t),\\
\end{eqnarray*}
where $k=0,1,\ldots,n$, which follows from \eqref{E:Jacobi+1-1}, the 
relation~\eqref{E:DiffJacobi} and Lemma~\ref{L:Base-G}.

As a consequence, we obtain the explicit formula for the dual $G^n_j$ 
polynomials.

\begin{theorem}\label{T:dual-G^n_i}
The dual polynomials $dG^n_j(t;\alpha,\beta)$ 
$(0\leq j\leq n;\,\alpha,\beta>-1)$ are represented in the Jacobi polynomial 
basis as follows:
\begin{equation}\label{E:dual-G^n_i-rep}
dG^n_j(t;\alpha,\beta)=
      \sum_{i=0}^{n}\frac{1}{h^{(\alpha,\beta)}_i}c^{(\alpha,\beta)}_{ji}
                             P^{(\alpha,\beta)}_{i}(t)\qquad (j=0,1,\ldots,n)
\end{equation}
(cf.~\eqref{E:Def-h_k}), and
$$
c^{(\alpha,\beta)}_{ji}:=\left\{
\begin{array}{lcl}
\displaystyle
\frac{1}{i!}((\alpha+1)_i+(-1)^i(\beta+1)_i)&:& (j=0),\\[1.5ex]
\displaystyle
\omega_j\frac{i+\sigma}{2}P^{(\alpha+1,\beta+1)}_{i-1}(\tau_j)
                        &:& (1\leq j\leq n)
\end{array}
\right.
$$
for $i=0,1,\ldots,n$ (cf.~the convention~\eqref{E:Convention}).
\end{theorem}

\begin{remark}\label{R:Clenshaw}
Note that from the representation~\eqref{E:dual-G^n_i-rep}, it follows that 
$dG^n_j(t;\alpha,\beta)$ can be evaluated for given $n$, $0\leq j\leq n$, 
$t\in\mathbb R$ and $\alpha,\beta>-1$ in $O(n)$ time using the well-known 
\textit{Clenshaw algorithm} (see, e.g., \cite{WG}).
\end{remark}

We show that there exists a closed-form expression for the dual polynomials 
$dG^n_j(t;\alpha,\beta)$ if $1\leq j\leq n$.

\begin{theorem}\label{T:dual-G^n_i-closed-form}
For $j=1,2,\ldots,n$, it holds that
\begin{equation}\label{E:dual-G^n_i-closed-form}
dG^n_j(t;\alpha,\beta)=2\omega_j\frac{(n+1)(n+\sigma)}
                                     {h^{(\alpha,\beta)}_n(2n+\sigma)_2}
  \frac{b^{(\alpha,\beta)}_{n+1,j}P^{(\alpha,\beta)}_{n}(t)
           -b^{(\alpha,\beta)}_{nj}P^{(\alpha,\beta)}_{n+1}(t)-
                                K^{(\alpha,\beta)}_n(\tau_j,t)}{\tau_j-t},
\end{equation}
where 
$b^{(\alpha,\beta)}_{ij}:=
 \frac{1}{2}(i+\sigma)P^{(\alpha+1,\beta+1)}_{i-1}(\tau_j)$ $(i=n,n+1)$, and
$$
K^{(\alpha,\beta)}_n(x,y):=
 \frac{P^{(\alpha,\beta)}_{n+1}(x)P^{(\alpha,\beta)}_{n}(y)-
       P^{(\alpha,\beta)}_{n}(x)P^{(\alpha,\beta)}_{n+1}(y)}{x-y}.
$$
\end{theorem}
\begin{proof}
From \eqref{E:Def-h_k}, \eqref{E:Jacobi-lc} and Theorem~\ref{T:Thm-LW2006-II}, 
we have
$$
2\frac{(n+1)(n+\sigma)}{h^{(\alpha,\beta)}_n(2n+\sigma)_2}
K^{(\alpha,\beta)}_n(x,y)=\sum_{i=0}^{n}G^n_i(x)dG^n_i(y;\alpha,\beta).
$$
Differentiating with respect to $x$ gives
$$
2\frac{(n+1)(n+\sigma)}{h^{(\alpha,\beta)}_n(2n+\sigma)_2}
\frac{\mathrm{d}}{\mathrm{d}x}K^{(\alpha,\beta)}_n(x,y)=
\sum_{i=1}^{n}\omega_i^{-1}L^n_i(x)dG^n_i(y;\alpha,\beta)
$$
(cf.~\eqref{E:Def-G^n_i}, \eqref{E:Def-L^n_i}, \eqref{E:L^n_i-G^n_i}).
Since $L^n_i(\tau_j)=\delta_{ij}$ for $1\leq i,j\leq n$, it follows that
$$
dG^n_j(t;\alpha,\beta)=2\omega_j\frac{(n+1)(n+\sigma)}
                                     {h^{(\alpha,\beta)}_n(2n+\sigma)_2}
  \frac{\mathrm{d}}{\mathrm{d}x}K^{(\alpha,\beta)}_n(x,t)\Big|_{x:=\tau_j}
\qquad (j=1,2,\ldots,n).
$$

Using~\eqref{E:DiffJacobi}, we obtain the desired result.
\end{proof}

Summarizing this part of the paper, note that certain simplifications in the 
representation~\eqref{E:dual-G^n_i-closed-form} occur in the Legendre case 
(i.e., for $\alpha=\beta=0$), because
$$
K^{(0,0)}_n(\tau_j,t)=\frac{P_{n+1}(\tau_j)P_n(t)}{\tau_j-t}\qquad
                                                              (0\leq j\leq n).
$$

\subsection{Dual Gauss-Legendre polynomials}           \label{SS:GL-dual-poly}

Proceeding similarly as in~\S\ref{SS:G^n_j-dual-poly}, we can obtain an 
explicit expression for the \textit{dual GL polynomials}
$dF^n_j(t;\alpha,\beta)\in\Pi_n$ $(0\leq j\leq n;\,\alpha,\beta>-1)$ satisfying 
the conditions:
$$
\left<F^n_i,dF^n_j(\cdot;\alpha,\beta)\right>_{\alpha,\beta}=
                                          \delta_{ij}\qquad (0\leq i,j\leq n)
$$
(cf.~\eqref{E:Def-F^n_i}).

Namely, the following theorem holds true. We omit its proof, as it is analogous
to the justification  of Theorem~\ref{T:dual-G^n_i}, however, we make use of
Corollary~\ref{C:Corollary-1} (see also~\cite[Theorem 7]{Moon2023}) in 
Theorem~\ref{T:Thm-LW2006-I} instead of Lemma~\ref{L:Base-G}.

\begin{theorem}\label{T:dual-F^n_i}
The dual GL polynomials $dF^n_j(t;\alpha,\beta)$ 
$(0\leq j\leq n;\,\alpha,\beta>-1)$ are represented in the Jacobi polynomial 
basis as follows:
\begin{equation*}
dF^n_j(t;\alpha,\beta)=
      \sum_{i=0}^{n}\frac{1}{h^{(\alpha,\beta)}_i}a^{(\alpha,\beta)}_{ji}
                             P^{(\alpha,\beta)}_{i}(t)\qquad (j=0,1,\ldots,n)
\end{equation*}
(cf.~\eqref{E:Def-h_k}), and
$$
a^{(\alpha,\beta)}_{ji}:=(-1)^i\frac{(\beta+1)_i}{i!}+\frac{i+\sigma}{2}
\sum_{\ell=1}^{j}\omega_\ell P^{(\alpha+1,\beta+1)}_{i-1}(\tau_\ell)
$$
for $0\leq i,j\leq n$ (cf.~the convention~\eqref{E:Convention}).
\end{theorem}

There is also a simple relation between the dual GL polynomials and the dual 
polynomials introduced in \S\ref{SS:G^n_j-dual-poly}.   

\begin{theorem}\label{T:dual-F^n_i-G^n_i}
The following identity holds:
$$
dF^n_j(t;\alpha,\beta)=dF^n_{j-1}(t;\alpha,\beta)+dG^n_{j}(t;\alpha,\beta),
$$
where $j=1,2,\ldots,n$.
\end{theorem}
\begin{proof}
It is enough to observe that
$$
a^{(\alpha,\beta)}_{ji}=a^{(\alpha,\beta)}_{j-1,i}+
           \frac{i+\sigma}{2}\omega_jP^{(\alpha+1,\beta+1)}_{j-1}(\tau_j)
$$
and use~\eqref{E:dual-G^n_i-rep} for $j=1,2,\ldots,n$.
\end{proof}

\begin{remark}
The last theorem implies that for given $n$, $0\leq j\leq n$, 
$t\in\mathbb R$ and $\alpha,\beta>-1$ the value $dF^n_j(t;\alpha,\beta)$
can be computed with $O(nj)$ computational complexity 
(cf.~Remark~\ref{R:Clenshaw}).
\end{remark}

\section{Applications}                                  \label{S:Applications}

Let us point out some possible applications of the results presented in this 
article. They are related to: \textit{(i)} the representations of the 
polynomials $G^n_i$ and $F^n_i$; \textit{(ii)} the degree reduction problem for 
polynomials and parametric curves; and \textit{(iii)} dual Lagrange bases.  

\subsection{Representations}                        \label{SS:Representations}

Recently, in~\cite{ChW2026}, new representations of the polynomials $G^n_i$ and 
$F^n_i$ (see~\eqref{E:Def-G^n_i}, \eqref{E:Def-F^n_i}) were derived. More 
precisely, in this article, the aforementioned polynomials, as well as their 
derivatives, were expressed in the \textit{shifted power} basis $(x+1)^j$ and 
the \textit{symmetric Jacobi} basis $P_j^{(\alpha,\alpha)}$ 
(cf.~\S\ref{SS:Jacobi}). Using these representations, it is possible to 
efficiently evaluate not only the polynomials $G^n_i$ and $F^n_i$, but also 
their linear combinations in one or many points and thus, in particular, render 
the $d$-dimensional GL curve (see~\eqref{E:GL-curves}).

We now show that using the dual bases one can give many other representations 
of polynomials $G^n_i$ (and consequently also of the polynomials $F^n_i$).

Let $b^n_0, b^n_1,\ldots,b^n_n$ be a basis of $\Pi_n$. Let 
$d^n_0, d^n_1,\ldots,d^n_n$ be its dual basis corresponding to the inner
product $\left<\cdot,\cdot\right>_{0,0}$ (see~\eqref{E:JacobiInnerP}), i.e.,
$$
\int_{-1}^{1}b^n_i(x)d^n_j(x)\,\mathrm{d}x=\delta_{ij}
$$
for $0\leq i,j\leq n$ (cf.~\S\ref{S:DualBases}). 

Then
$$
(n+1)\frac{P_{n+1}(x)P_n(t)-P_n(x)P_{n+1}(t)}{x-t}=
                                     \sum_{j=0}^{n}d^n_j(t)b^n_j(x),
$$
where $P_k$ is the $k$th Legendre polynomial $(k=n,n+1)$ 
(cf.~Theorem~\ref{T:Thm-LW2006-II} and \eqref{E:Def-h_k}, \eqref{E:Jacobi-lc}).

Setting $t:=\tau_i$ and remembering that $P_n(\tau_i)=0$ for $i=1,2,\ldots,n$
(cf.~Assumption~\ref{A:Assumption-1}), we obtain
$$
\frac{P_n(x)}{x-\tau_i}=-\frac{1}{(n+1)P_{n+1}(\tau_i)}
                                     \sum_{j=0}^{n}d^n_j(\tau_i)b^n_j(x).
$$
Integrating the obtained identity over $x\in[-1,t]$, where $-1\leq t\leq 1$, 
and using~\cite[Eq.~(4.10)]{ChW2026} gives
$$
G^n_i(t)=-\frac{1}{2}+\frac{1}{2}\sum_{j=0}^{n}d^n_j(\tau_i)g^n_j(x),
$$
where
$$
g^n_j(x):=\int_{-1}^{t}b^n_j(x)\,\mathrm{d}x\qquad(j=0,1,\ldots,n).
$$

Summarizing, the coefficients of the expansion of the
polynomials~\eqref{E:Def-G^n_i} in the $g^n_j$ basis are, in fact, evaluations 
of the elements of the dual basis to the basis $b^n_j$ at zeros of the $n$th
Legendre polynomial. In this context, the formulas presented in~\cite{ChW2026} 
can also be interpreted in this way.

\subsection{Degree reduction}                                \label{SS:DegRed}

The \textit{degree reduction} problem often appears in practical applications, 
for example in approximation theory, numerical analysis, and CAGD, as it is 
related to data compression and data exchange between different CAD or 
computational systems.

In particular, the problem of \textit{constrained degree reduction of 
B\'{e}zier curves} has been extensively studied in recent years. See, e.g., 
\cite{WL2009} (and the references therein), where a technique using the 
so-called \textit{constrained dual Bernstein polynomials} was applied, 
resulting in an algorithm with low computational complexity and good numerical 
properties. 

Let us consider the degree reduction problems for polynomials given in the 
$G^n_i$ and $F^n_i$ bases,
$$
w_n(t):=\sum_{i=0}^{n}g_iG^n_i(t),\qquad
v_n(t):=\sum_{i=0}^{n}f_iF^n_i(t)\qquad (g_i,f_i\in\mathbb R).
$$
The goal is to find polynomials $w^\ast_m,v^\ast_m\in\Pi_m$, where $m<n$, 
satisfying
$$
||w_n-w^\ast_m||_{\alpha,\beta}=\min_{w\in\Pi_m}||w_n-w||_{\alpha,\beta},
\qquad
||v_n-v^\ast_m||_{\alpha,\beta}=\min_{v\in\Pi_m}||v_n-v||_{\alpha,\beta}.
$$
Here
$$
||f||_{\alpha,\beta}:=
     \left(\int_{-1}^{1}(1-x)^\alpha(1+x)^{\beta}f^2(x)\;
                    \mathrm{d}x\right)^{\frac{1}{2}}\qquad (\alpha,\beta>-1)
$$
(cf.~\eqref{E:JacobiInnerP}).

It is well-known (see Section~\ref{S:DualBases}) that
$$
w^\ast_m(t)=\sum_{j=0}^{m}g^\ast_jG^n_j(t), 
$$
where
$$
g^\ast_j:=\left<w_n,dG^m_j\right>_{\alpha,\beta}=\sum_{i=0}^{n}g_i
\left<G^n_i,dG^m_j\right>_{\alpha,\beta}\qquad (0\leq j\leq m),
$$
as well as
$$
v^\ast_m(t)=\sum_{j=0}^{m}f^\ast_jF^n_j(t), 
$$
where
$$
f^\ast_j:=\left<v_n,dF^m_j\right>_{\alpha,\beta}=\sum_{i=0}^{n}f_i
\left<F^n_i,dF^m_j\right>_{\alpha,\beta}\qquad (0\leq j\leq m).
$$

The integrals $\left<G^n_i,dG^m_j\right>_{\alpha,\beta}$ and 
$\left<F^n_i,dF^m_j\right>_{\alpha,\beta}$ $(0\leq i\leq n,\ 0\leq j\leq m)$
can be computed, for example, using appropriate Gauss--Jacobi quadratures rules 
with $n$ nodes as $m<n$. In particular, if $\alpha=\beta=0$, then the 
quadrature rule~\eqref{E:GL-quadrature} can be applied (cf.~also 
Remark~\ref{R:GL-quadrature}).

Note that the considered problems are, among others, closely related to the 
degree reduction problem of the parametric GL curves 
${\sf P}_n:[-1,1]\to{\mathbb E}^d$ $(d\in\mathbb N)$ 
introduced in~\cite{Moon2023},
\begin{equation}\label{E:GL-curves}
{\sf P}_n(t):=\sum_{k=0}^{n}{\sf W}_kF^n_k(t)\qquad 
({\sf W}_0,{\sf W}_1,\ldots,{\sf W}_n\in{\mathbb E}^d).
\end{equation}

\subsection{Dual bases for Lagrange polynomials} \label{SS:DualBases-Lagrange}

Observe that the technique used in the proof of 
Theorem~\ref{T:dual-G^n_i-closed-form} allows us to derive a dual basis
for the \textit{Lagrange polynomial basis} with respect to a chosen inner
product.

Let $m\in\mathbb N$ and $x_0,x_1,\ldots,x_m\in\mathbb R$ be distinct points. We 
define the Lagrange polynomial basis as follows:
$$
\lambda^m_i(x):=\prod_{j=0 \atop j\neq i}^{m}\frac{x-x_j}{x_i-x_j}
                                                      \qquad(i=0,1,\ldots,m)
$$
(cf.~\eqref{E:Def-L^n_i}). Certainly, 
$\mathrm{span}\{\lambda^m_i\,:\,0\leq i\leq m\}=\Pi_m$.

Let us fix an inner product 
$\left<\cdot,\cdot\right>:\Pi_m\times\Pi_m\to\mathbb R$ and let 
$d\lambda^m_0,d\lambda^m_1,\ldots,d\lambda^m_m\in\Pi_m$ be the dual basis for 
the Lagrange polynomial basis with respect to the given inner product. This
means that
$$
\left<\lambda^m_i,d\lambda^m_j\right>=\delta_{ij}
$$
for $0\leq i,j\leq m$ (see Section~\ref{S:DualBases}). 

We can give a closed-form expression for the dual polynomials $d\lambda^m_i$
in terms of the orthogonal polynomials corresponding to the inner product
$\left<\cdot,\cdot\right>$.

\begin{theorem}
The dual polynomials $d\lambda^m_j$ have the following form:
\begin{eqnarray}
\label{E:Proof-I}
d\lambda^m_j(t)&=&\frac{\ell_m}{\ell_{m+1}h_m}
                         \frac{q_{m+1}(x_j)q_m(t)-q_m(x_j)q_{m+1}(t)}{x_j-t}\\
\label{E:Proof-II}                         
               &=&\sum_{i=0}^{m}h_i^{-1}q_i(x_j)q_i(t),
\end{eqnarray}
where $j=0,1,\ldots,m$ and $\ell_m$ and $\ell_{m+1}$ denote the leading 
coefficients of the polynomials $q_m$ and $q_{m+1}$, respectively. Here
polynomials $q_0,q_1,\ldots,q_{m+1}$ form the orthogonal system with respect to
the inner product $\left<\cdot,\cdot\right>$,
$$
\left<q_k,q_\ell\right>=\delta_{k\ell}h_k
                                      \qquad (k,\ell=0,1,\ldots,m+1;\;h_k>0),
$$
and $q_k\in\Pi_{k}\setminus\Pi_{k-1}$ $(0\leq k\leq m+1)$.
\end{theorem}
\begin{proof}
From Theorem~\ref{T:Thm-LW2006-II}, we have
$$
\frac{\ell_m}{\ell_{m+1}h_m}\frac{q_{m+1}(x)q_m(t)-q_m(x)q_{m+1}(t)}{x-t}=
\sum_{i=0}^{m}\lambda^m_i(x)d\lambda^m_i(t).
$$
Setting $x:=x_j$ $(j=0,1,\ldots,m)$ and taking into account that 
$\lambda^m_i(x_{j})=\delta_{ij}$, we obtain~\eqref{E:Proof-I}. 

Equation~\eqref{E:Proof-II} follows from the well-known representation of the
Christoffel–Darboux kernel~\eqref{E:Ch-D-kernel} in the appropriate orthogonal 
basis.
\end{proof}

\section{Conclusions and future work}                    \label{S:Conclusions}

We derive explicit representations for two new families of dual polynomials 
associated with the GL polynomials and the Jacobi inner product. These dual 
polynomials can be used, in particular, for the solution of certain 
approximation problems related to the degree reduction of parametric curves. 
We also propose a general method for deriving new representations of the GL 
polynomials and constructing the Lagrange dual bases.

It seems possible to derive a recurrence relation for the dual polynomials 
$dG^n_j$ with respect to $j$. Using such a recurrence, it will probably be 
possible to evaluate all values $dG^n_0(t;\alpha,\beta), 
dG^n_1(t;\alpha,\beta),\ldots,dG^n_n(t;\alpha,\beta)$ for a given 
$t\in\mathbb R$ and $\alpha,\beta>-1$ faster than in $O(n^2)$ time 
(cf.~Remark~\ref{R:Clenshaw}).

It would be interesting to propose a method for the degree reduction of GL 
curves~\eqref{E:GL-curves}, which is similar to that described in~\cite{WL2009} 
for parametric B\'{e}zier curves, where some nonstandard properties of the 
(constrained) dual Bernstein polynomials and their relations with 
the shifted Jacobi and Hahn orthogonal polynomials were used. Such a method 
would probably have lower computational complexity than the one proposed in 
Section~\ref{SS:DegRed}, but it seems to be quite challenging to derive an 
efficient algorithm in full detail. 

We leave these issues for future research, as it requires further study of the 
properties of the new dual polynomial families introduced in this article.

\bibliographystyle{elsart-num-sort}
\biboptions{sort&compress}
\bibliography{GL-dual-basis-rev}

\begin{thebibliography}{10}
\expandafter\ifx\csname url\endcsname\relax
  \def\url#1{\texttt{#1}}\fi
\expandafter\ifx\csname urlprefix\endcsname\relax\def\urlprefix{URL }\fi

\bibitem{Agoston2005}
M.~K. Agoston, Computer graphics and geometric modelling. Implementation \&
  algorithms, 1st ed., Springer, London, 2005.

\bibitem{BJ1998}
{B. J\"{u}ttler}, {The dual basis functions of the Bernstein polynomials},
  Advances in Computational Mathematics 8 (1998) 345--352.

\bibitem{BJ2007}
M.~Barto\v{n}, {B. J\"{u}ttler}, {Computing roots of polynomials by quadratic
  clipping}, Computer Aided Geometric Design 24 (2007) 125--141.

\bibitem{ChW2026}
F.~Chudy, P.~Wo\'{z}ny, Evaluation of {G}auss--{L}egendre curves, 2026,
  \url{http://arxiv.org/abs/2604.17331}.

\bibitem{ZC1987}
Z.~Ciesielski, {The basis of B-splines in the space of algebraic polynomials},
  Ukrainian Mathematical Journal 38 (1987) 311--315.

\bibitem{Farin2002}
G.~Farin, Curves and surfaces for {C}omputer-{A}ided {G}eometric {D}esign. A
  practical guide, 5th ed., Academic Press, Boston, 2002.

\bibitem{WG}
W.~Gautschi, Numerical analysis, 2nd ed., Birkh\"{a}user, 2011.

\bibitem{RNG}
R.~N. Goldman, {Dual polynomial bases}, Journal of Approximation Theory 79
  (1994) 311--346.

\bibitem{GXZ06}
Z.~Guohui, L.~Xiuping, S.~Zhixun, {A dual functional to the univariate
  B-spline}, Journal of Computational and Applied Mathematics 195 (2006)
  292--299.

\bibitem{HT2013}
N.~Hale, A.~Townsend, Fast and accurate computation of {G}auss--{L}egendre and
  {G}auss--{J}acobi quadrature nodes and weights, SIAM Journal on Scientific
  Computing 35 (2013) A652--A674.

\bibitem{JM2018}
F.~Johansson, M.~Mezzarobba, Fast and rigorous arbitrary-precision computation
  of {G}auss--{L}egendre quadrature nodes and weights, SIAM Journal on
  Scientific Computing 40 (2018) C726--C747.

\bibitem{Kim2026}
S.~H. Kim, H.-P. Schr\"{o}cker, H.~P. Moon, Fundamental geometric operations
  for {G}auss--{L}egendre curves via linear transformations of control edges,
  Computer Aided Geometric Design 128 (2026) 102575.

\bibitem{KLS2010}
R.~Koekoek, P.~A. Lesky, R.~F. Swarttouw, Hypergeometric orthogonal polynomials
  and their $q$-analogues, Springer Monographs in Mathematics, Springer Berlin
  Heidelberg, 2010.

\bibitem{LW2006b}
S.~Lewanowicz, P.~Wo\'{z}ny, {Connections between two-variable Bernstein and
  Jacobi polynomials on the triangle}, Journal of Computational and Applied
  Mathematics 197 (2006) 520--533.

\bibitem{LW2006}
S.~Lewanowicz, P.~Wo\'{z}ny, Dual generalized {B}ernstein basis, Journal of
  Approximation Theory 138 (2006) 129--150.

\bibitem{LW2011}
S.~Lewanowicz, P.~Wo\'{z}ny, {B\'{e}zier representation of the constrained dual
  Bernstein polynomials}, Applied Mathematics and Computation 218 (2011)
  4580--4586.

\bibitem{LW2011b}
S.~Lewanowicz, P.~Wo\'{z}ny, {Multi-degree reduction of tensor product
  B\'{e}zier surfaces with general boundary constraints}, Applied Mathematics
  and Computation 217 (2011) 4596--4611.

\bibitem{LWK2011}
S.~Lewanowicz, P.~Wo\'{z}ny, P.~Keller, {Polynomial approximation of rational
  B\'{e}zier curves with constraints}, Numerical Algorithms 59 (2012) 607--622.

\bibitem{LZLW2009}
L.~Liu, L.~Zhang, B.~Lin, G.~Wang, {Fast approach for computing roots of
  polynomials using cubic clipping}, Computer Aided Geometric Design 26 (2009)
  547--559.

\bibitem{Moon2023}
H.~P. Moon, S.~H. Kim, S.-H. Kwon, {G}auss--{L}egendre polynomial basis for the
  shape control of polynomial curves, Applied Mathematics and Computation 451
  (2023) 127995.

\bibitem{Moon2026}
H.~P. Moon, S.~H. Kim, S.-H. Kwon, A novel method for manipulating polynomial
  curves by the {G}auss--{L}egendre control polygon with points interpolating
  property, Applied Mathematics and Computation 512 (2026) 129760.

\bibitem{RN2007}
A.~Rababah, M.~Al-Natour, {The weighted dual functionals for the univariate
  Bernstein basis}, Applied Mathematics and Computation 186 (2007) 1581--1590.

\bibitem{RN2008}
A.~Rababah, M.~Al-Natour, {Weighted dual functions for Bernstein basis
  satisfying boundary conditions}, Applied Mathematics and Computation 199
  (2008) 1581--1590.

\bibitem{RH2026}
A.~Ramanantoanina, K.~Hormann, G{L}-$k$ curves: a family of polynomial curves
  for intuitive modelling, Computer Aided Geometric Design 127 (2026) 102558.

\bibitem{STW2011}
J.~Shen, T.~Tang, L.-L. Wang, Orthogonal polynomials and related approximation
  results, in: Spectral methods: algorithms, analysis and applications, vol.~41
  of Springer Series in Computational Mathematics, Springer Berlin Heidelberg,
  2011.

\bibitem{W2012}
P.~Wo\'{z}ny, {Simple algorithms for computing the B\'{e}zier coefficients of
  the constrained dual Bernstein polynomials}, Applied Mathematics and
  Computation (2012){, accepted for publication}.

\bibitem{PW2013}
P.~Wo\'{z}ny, Construction of dual bases, Journal of Computational and Applied
  Mathematics 245 (2013) 75--85.

\bibitem{PW2014}
P.~Wo\'{z}ny, Construction of dual {B}-spline functions, Journal of
  Computational and Applied Mathematics 260 (2014) 301--311.

\bibitem{WL2009}
P.~Wo\'{z}ny, S.~Lewanowicz, Multi-degree reduction of {B}\'{e}zier curves with
  constraints, using dual {B}ernstein basis polynomials, Computer Aided
  Geometric Design 26 (2009) 566--579.

\bibitem{WL2010}
P.~Wo\'{z}ny, S.~Lewanowicz, {Constrained multi-degree reduction of triangular
  B\'{e}zier surfaces, using dual Bernstein polynomials}, Journal of
  Computational and Applied Mathematics 235 (2010) 785--804.

\bibitem{Zhang2010}
L.~Zhang, J.~Tan, Z.~Dong, {The dual bases for the B\'{e}zier-Said-Wang type
  generalized Ball polynomial bases and their applications}, Applied
  Mathematics and Computation 217 (2010) 3088--3101.

\bibitem{Zhang2009a}
L.~Zhang, J.~Tan, H.~Wu, Z.~Liu, {The weighted dual functions for
  Wang-B\'{e}zier type generalized Ball bases and their applications}, Applied
  Mathematics and Computation 215 (2009) 22--36.

\bibitem{Zhang2009b}
L.~Zhang, H.~Wu, J.~Tan, {Dual bases for Wang-B\'{e}zier basis and their
  applications}, Applied Mathematics and Computation 214 (2009) 218--227.

\bibitem{Zhang2009c}
L.~Zhang, H.~Wu, J.~Tan, {Dual basis functions for the NS power and their
  applications}, Applied Mathematics and Computation 207 (2009) 434--441.

\end{thebibliography}

\end{document}